\documentclass[11pt,reqno,british,a4paper,twoside]{amsart}

\usepackage{mathtools,amssymb,enumitem,tikz}
\usepackage{csquotes,microtype}
\usepackage[hidelinks]{hyperref}

\usetikzlibrary{calc,intersections,positioning}

\mathtoolsset{mathic=true}
\setlist[enumerate,1]{label=\upshape(\roman*)}

\AddToHook{bfseries/default}{\boldmath}  
\AddToHook{mdseries/default}{\unboldmath}  

\theoremstyle{plain}
\newtheorem{theorem}{Theorem}[section]
\newtheorem{proposition}[theorem]{Proposition}
\newtheorem{lemma}[theorem]{Lemma}

\theoremstyle{definition}

\newtheorem*{acknowledgements}{Acknowledgements}

\theoremstyle{remark}
\newtheorem{remark}[theorem]{Remark}
\newtheorem{example}[theorem]{Example}

\numberwithin{equation}{section}

\NewDocumentCommand{\any}{}{\,\cdot\,}
\NewDocumentCommand{\hook}{}{\mathbin{\lrcorner}}

\NewDocumentCommand{\Lie}{ m }{\operatorname{#1}}
\NewDocumentCommand{\lie}{ m }{\operatorname{\mathfrak{#1}}}

\NewDocumentCommand{\GL}{}{\Lie{GL}}

\NewDocumentCommand{\SU}{}{\Lie{SU}}
\NewDocumentCommand{\su}{}{\lie{su}}

\NewDocumentCommand{\la}{ m }{\mathsf{#1}}

\NewDocumentCommand{\Hodge}{}{\mathop{}\!{*}}

\NewDocumentCommand{\bC}{}{\mathbb{C}}
\NewDocumentCommand{\bR}{}{\mathbb{R}}

\NewDocumentCommand{\bZ}{}{\mathbb{Z}}

\NewDocumentCommand{\CP}{}{\bC P}

\NewDocumentCommand{\cC}{}{\mathcal{C}}
\NewDocumentCommand{\cE}{}{\mathcal{E}}

\NewDocumentCommand{\point}{m}{\mathrm{#1}}
\NewDocumentCommand{\pA}{}{\point{A}}
\NewDocumentCommand{\pB}{}{\point{B}}
\NewDocumentCommand{\pC}{}{\point{C}}
\NewDocumentCommand{\pD}{}{\point{D}}

\DeclareMathOperator{\LIE}{Lie}
\DeclareMathOperator{\diag}{diag}

\NewDocumentCommand{\Id}{}{\mathrm{Id}}

\DeclarePairedDelimiter{\Span}{\langle}{\rangle}
\DeclarePairedDelimiter{\abs}{\lvert}{\rvert}
\DeclarePairedDelimiterX{\inp}[2]{\langle}{\rangle}{#1, #2}

\DeclarePairedDelimiter{\paren}{\lparen}{\rparen}

\NewDocumentCommand{\with}{}{\mid}
\NewDocumentCommand{\SetSymbol}{o}{\nonscript\:#1|
  \allowbreak\nonscript\:\mathopen{}}
\DeclarePairedDelimiterX{\Set}[1]{\lbrace}{\rbrace}{%
  \renewcommand{\with}{\SetSymbol[\delimsize]}%
#1 }

\NewDocumentCommand{\Fd}{m m}{\frac{d #1}{d #2}}
\NewDocumentCommand{\Fdt}{O{2} m m}{\frac{d^#1 #2}{d #3^#1}}

\NewDocumentCommand{\eqbreak}{O{2}}{\\&\hspace{#1em}}
\NewDocumentCommand{\eqand}{O{1}}{\hspace{#1em}\text{and}\hspace{#1em}}

\begin{document}

\title{Cohomological lifting of multi-toric graphs}

\author{Kael Dixon}

\email{kael.dix@gmail.com}

\author{Thomas Bruun Madsen}

\address{School of Computing and Engineering\\
University of West London\\
St. Mary’s Road\\
Ealing\\
London W5 5RF\\
United Kingdom}

\email{thomas.madsen@uwl.ac.uk}

\author{Andrew Swann}

\address{Department of Mathematics and DIGIT\\
Aarhus University\\
Ny Munkegade 118, Bldg 1530\\
DK-8000 Aarhus C\\
Denmark}

\email{swann@math.au.dk}

\begin{abstract}
  We study \( G_{2} \)-manifolds obtained from circle bundles over symplectic
  \( \SU(3) \)-manifolds with \( T^{2} \)-symmetry.
  When the geometry is multi-Hamiltonian, we show how the compact part of the
  resulting multi-moment graph for the \( G_{2} \)-structure may obtained
  cohomologically from the base.
  The lifting procedure is illustrated in the context of toric geometry.
\end{abstract}


\maketitle

\section{Introduction}
\label{sec:introduction}

Symmetry is a powerful tool to understand different classes of geometric
structures.
For a manifold~\( M \) with a closed differential form preserved by a Lie
group~\( G \), the notion of multi-moment was introduced in
\cite{Madsen-S:multi-moment,Madsen-S:closed}.
If \( G \) is Abelian, this is a smooth invariant map \( \nu\colon M \to V \), for a
certain finite-dimensional vector space~\( V \), which generalises the usual
notion of moment map~\( \mu\colon M \to \lie{g}^{*} \) in symplectic geometry.
We will study this map for special geometries in dimensions \( 6 \) and~\( 7 \)
with torus symmetry.
Here, there is a trivalent graph~\( \Gamma \) immersed in~\( V \) given by the image
under \( \nu \) of the points with non-trivial stabiliser.
These immersions have edges that are line segments, and satisfy a zero-tension
condition at each vertex.
For Calabi-Yau three-folds these graphs are often used in the physics literature,
see Bouchard~\cite{Bouchard:CY}, where they are planar.
In \cite{Madsen-S:toric-G2}, we showed that for \( 7 \)-manifolds with
holonomy~\( G_{2} \), there are also such graphs, but now the graphs live in
three-dimensions.

Recently, Foscolo, Haskins and Nordström \cite{Foscolo-HN:from-CY} gave a construction many examples of
manifolds of holonomy~\( G_{2} \) as circle bundles over Calabi-Yau three-folds
with certain asymptotics; a powerful result, that significantly increases the
number of known non-compact \( G_{2} \)-examples, thanks to a wealth of
information from the Calabi-Yau literature.
This motivates us to study how graphs may be lifted from \( 6 \)-dimensional
geometries to those in dimension~\( 7 \).
One of our main results, see Section~\ref{sec:graph-lifting}, is that for
compact edges of the graph, this information is determined cohomologically by
the classes of the symplectic structure on the base and the curvature form of
the circle bundle.
We derive the statements in a broader setting than that of reduced holonomy,
requiring just that the \( G_{2} \) three-form is closed and by using the
sympletic property of the base.
Some generality in this direction is needed to apply the results to the
construction of \cite{Foscolo-HN:from-CY}, since there the base Calabi-Yau
structure is deformed to a more general \( \SU(3) \)-structure before getting a
submersion structure, with just the group action and sympletic form surviving
the process.

In Section~\ref{sec:toric-CY}, we specialise to the case when the base is a toric
Calabi-Yau manifold.
There we can use data from the associated fans, and show that loops in the
\( \SU(3) \)-graphs lift to closed loops in the \( G_{2} \)-graphs precisely
when a cohomological condition in~\cite{Foscolo-HN:from-CY} is satisfied.
This indicates that the condition is intrinsic to the general set-up and not
purely an artefact of the adiabatic limit construction used
in~\cite{Foscolo-HN:from-CY}.
We close the paper with a number of explicit examples, demonstrating the utility
of the lifting formulae.

\begin{acknowledgements}
  We are grateful for partial support from the Aarhus University Research
  Foundation.
  We thank Lorenzo Foscolo for useful discussions of results in this area.
\end{acknowledgements}

\section{Special multi-toric geometries}
\label{sec:special-multi-toric}

A \( G_{2} \)-structure on a seven-dimensional manifold \( M \) is given by a
three-form \( \varphi \) that is pointwise linearly equivalent to
\begin{equation*}
  e^{123} - e^{147} - e^{167} - e^{246} - e^{275} - e^{347} - e^{356},
\end{equation*}
where \( e^{i} \) is a basis for \( (\bR^{7})^{*} \cong T_{p}^{*}M \) and
\( e^{abc} = e^{a} \wedge e^{b} \wedge e^{c} \).
This basis is orthonormal for a Riemannian metric~\( g \) and specifies an
orientation, so also a Hodge \( \Hodge \)-operator; these objects depend only
on~\( \varphi \), see Bryant~\cite{Bryant:exceptional}.
The \( G_{2} \)-structure is \emph{closed} if \( d\varphi = 0 \), \cite{Cleyton-S:G2}.
If both \( d\varphi = 0 = d \Hodge \varphi \), then the (restricted) holonomy
of~\( g \) is contained in~\( G_{2} \), and the structure is called
\emph{torsion-free}.

Suppose a three-dimensional torus \( T^{3} = \bR^{3}/\bZ^{3} \) acts effectively
on~\( M \) preserving~\( \varphi \), and hence \( g \) and~\( \Hodge \varphi \).
If \( \la{U}_{1},\la{U}_{2},\la{U}_{3} \) is a basis for \( \LIE(T^{3}) \), then
we write \( U_{1},U_{2},U_{3} \) for the corresponding vector fields generated
on~\( M \).
When \( d\varphi = 0 \), repeated use of the Cartan formula gives that the
one-form~\( \varphi(U_{i},U_{j},\any) \) is closed.
A closed \( G_{2} \)-structure is \emph{multi-Hamiltonian} if there is an
invariant function
\( \nu = (\nu_{1},\nu_{2},\nu_{3}) \colon M \to \bR^{3} = \Lambda^{2}\LIE(T^{3})^{*} \) such
that
\begin{equation}
  \label{eq:d-nu}
  d\nu_{i} = \varphi(U_{j},U_{k},\any) \qquad\text{for \( (ijk) = (123) \) as permuations.}
\end{equation}
We call \( \nu \) a \emph{multi-moment map} and note that \( \nu \) is guaranteed to
exist if \( b_{1}(M) = 0 \).
In \cite[Lemmas 2.5 and 2.6]{Madsen-S:toric-G2}, we showed that
\( \varphi(U_{1},U_{2},U_{3}) = 0 \) and that each isotropy is a connected subtorus of
dimension at most~\( 2 \).

If, in addition, the geometry is torsion-free, then there is also a multi-moment
map \( \hat{\nu} \) associated to~\( \Hodge \varphi \) satisfying
\( d\hat{\nu} = \Hodge \varphi(U_{1},U_{2},U_{3},\any) \).
It is notable, see \cite[Theorem~4.5]{Madsen-S:toric-G2}, that
\( (\nu,\hat{\nu}) \) induces a local homeomorphism of the orbit space
\( M/T^{3} \) to~\( \bR^{4} \).

Suppose now that \( \la{U}_{3} \) generates a free circle-action on~\( M \) of
fundamental period~\( 1 \).  We set
\begin{equation*}
   \theta = \frac{g(U_{3},\any)}{g(U_{3},U_{3})},
\end{equation*}
so \( \theta(U_{3}) = 1 \) and \( \theta \) is \( T^{3} \)-invariant.
Put \( B = M/\Span{U_{3}} \) and write \( \pi\colon M \to B \) for the projection.
From Apostolov and Salamon \cite[cf. Proposition~1]{Apostolov-S:K-G2}, we get the decompositions
\begin{equation}
  \label{eq:G2-decomposition}
  \begin{gathered}
    g = \theta^{2}\pi^{*}(h^{-1}) + \pi^{*}(h^{1/2}g_{B}),\\
    \varphi = \theta \wedge \pi^{*}(\omega) + \pi^{*}(h^{3/4}\psi_{+}),\\
    \Hodge \varphi = - \theta \wedge \pi^{*}(h^{1/4}\psi_{-}) + \pi^{*}(\tfrac{1}{2}h \omega^{2}),
  \end{gathered}
\end{equation}
where \( \pi^{*}h = g(U_{3},U_{3})^{-1} > 0 \) is smooth, and
\( (g_{B},\omega,\psi_{c} = \psi_{+}+i\psi_{-}) \) defines an
\( \SU(3) \)-structure on~\( B \).

Note that \( \pi^{*}(h) \) is also \( T^{3} \)-invariant, so the
\( \SU(3) \)-structure on \( B \) is preserved by the two-torus
\( T^{2} = T^{3}/\Span{\exp\la{U}_{3}} \).
If the \( G_{2} \)-structure is closed, then we have
\begin{equation*}
  0 = d\varphi = \theta \wedge \pi^{*}(d\omega) + \pi^{*}(F \wedge \omega + d(h^{3/2}\psi_{+})),
\end{equation*}
where \( \pi^{*}F = d\theta \) is the curvature of the circle bundle
\( M \to B \).
Contracting with \( U_{3} \), we see that \( d\omega = 0 \), so the
\( \SU(3) \)-structure is \emph{symplectic}.

We observe that
\begin{gather*}
  d\nu_{1} = \varphi(U_{2},U_{3},\any) = - (\pi^{*}\omega)(U_{2},\any),\\
  d\nu_{2} = \varphi(U_{3},U_{1},\any) = (\pi^{*}\omega)(U_{1},\any).
\end{gather*}
So putting
\begin{equation*}
  X_{1} = -\pi_{*}(U_{2}),\quad X_{2} = \pi_{*}(U_{1}),
\end{equation*}
we have a symplectic moment map
\( \mu = (\mu_{1},\mu_{2}) \colon B \to \bR^{2} \cong \LIE(T^{2})^{*} \) given by
\( \pi^{*}\mu = (\nu_{1},\nu_{2}) \), that is, \( \mu_{i} \)~are invariant functions
satisfying
\begin{equation}
  \label{eq:symplectic-mu}
  d\mu_{i} = \omega(X_{i},\any).
\end{equation}

There is a weak converse: given an \( \SU(3) \)-structure
\( (g_{B},\omega,\psi_{c}) \) on a six-manifold~\( B \), then
\( M = S^{1} \times B \) carries a product \( G_{2} \)-structure built as
in~\eqref{eq:G2-decomposition}, with \( \theta = ds \), for \( s \) the standard
parameter on \( S^{1} = \bR/\bZ \), and \( h \equiv 1 \).
This product structure is closed if and only if \( B \)~is \emph{symplectic
half-flat}, meaning \( d\omega = 0 = d\psi_{+} \), \cite{Conti-T:symplectic}, and has
holonomy (strictly) contained in \( G_{2} \) if and only if \( B \)~is
\emph{Calabi-Yau}, meaning \( d\omega = 0 = d\psi_{c} \).
Taking \( U_{1} = X_{2} \), \( U_{2} = -X_{1} \), \( U_{3} = \partial_{s} \), we may
transfer terminology from above for~\( M \) to~\( B \).
In particular, a symplectic \( B \) is \emph{multi-Hamiltonian} for a two-torus
action generated by \( X_{1},X_{2} \) preserving the \( \SU(3) \)-structure, if
there are invariant functions \( \mu_{1},\mu_{2} \)
satisfying~\eqref{eq:symplectic-mu}.
It follows that all stabilisers are connected sub-tori, and that
\( \omega(X_{1},X_{2}) = 0 \).
If \( B \) is symplectic half-flat, then at least locally we also have an
invariant function~\( \hat{\mu}_{+} \) with
\( d\hat{\mu}_{+} = \psi_{+}(X_{1},X_{2},\any) = d\nu_{3} \).
The functions \( (\mu_{1},\mu_{2},\hat{\mu}_{+}) \) are guaranteed to exist globally
if \( b_{1}(B) = 0 \).
If \( B \) is also Calabi-Yau, then there is an additional multi-moment map
\( \hat{\mu}_{-} \) satisfying
\( d\hat{\mu}_{-} = \psi_{-}(X_{1},X_{2},\any) = d\hat{\nu} \).
Again, we have that
\( (\mu,\hat{\mu}) = (\mu_{1},\mu_{2},\hat{\mu}_{+},\hat{\mu}_{-}) \) induces a local
homeomorphism \( B/T^{2} \to \bR^{4} \).

A significant source of Calabi-Yau structures is available within the category
of toric six-manifolds~\( B \).
These are connected and have a Hamiltonian action of a three-torus \( T^{3} \)
preserving the symplectic form~\( \omega \) and metric~\( g \), and consequently the
complex structure given by \( \omega = g(J\any,\any) \).
The complex volume form~\( \psi_{c} \) is a parallel section
of~\( \Lambda^{3,0}(B) \) of fixed length, so any other
section~\( \tilde{\psi}_{c} \) of \( \Lambda^{3,0}(B) \) of this length has the form
\( \tilde{\psi}_{c} = e^{2\pi i t}\psi_{c} \) for some local function
\( t \colon B \to \bR \).
If \( \tilde{\psi}_{c} \) is parallel then
\( 0 = \nabla\tilde{\psi}_{c} = 2\pi i\, e^{2\pi i t} dt \otimes \psi_{c} \), so
\( t \)~is constant and globally defined modulo~\( 1 \).
The action of~\( T^{3} \) on~\( B \) induces via pull-back an action on the
parallel sections of~\( \Lambda^{3,0}(B) \) of fixed length and hence gives a group
homomorphism \( T^{3} \to S^{1} \).
The kernel of this action as compact of dimension~\( 2 \), so its component of
the identity is a two-torus~\( T^{2} \) that acts on~\( B \) preserving the
\( \SU(3) \)-structure.
(See Li et al.\ \cite[Section~3.1]{Li-LLZ:top-vertex} for an algebraic version of
this.)
If \( b_{1}(B) = 0 \), we see that the \( T^{2} \)-action is multi-Hamiltonian.

A wide class of examples is given by toric asymptotically conical (AC)
Calabi-Yau three-folds.
These are asymptotic, with polynomial decay rates, to cones over toric
Sasaki-Einstein manifolds of dimension five, as constructed from good polytopes
by Cho, Futaki and Ono \cite{Cho-FO:toric-SE}.
Taking crepant resolutions of these cones yields smooth AC Calabi-Yau manifolds
with \( T^{3} \)-symmetry, classified by
Conlon and Hein \cite[Theorem~4.3]{Conlon-H:acCY-class}.
These structures are uniquely specified by the cohomology class of their Kähler
form.
Note that Li~\cite{Li:SYZ-Calabi-Yau} provides an example of a non-toric
Calabi-Yau manifold, with different asymptotics, that has a multi-Hamiltonian
action of~\( T^{2} \).

Suppose \( B \) is an AC Calabi-Yau three-fold with Kähler
form~\( \omega_{0} \), and that \( M \to B \) is a circle bundle over \( B \) with
curvature~\( F \) such that \( [F] \ne 0 \) in \( H^{2}(B) \) and
\( [\omega_{0}] \cup [F] = 0 \) in~\( H^{4}(B) \).
Then Foscolo, Haskins and Nordström \cite{Foscolo-HN:from-CY} construct metrics of holonomy~\( G_{2} \)
on~\( M \) which are asymptotically locally conical.
These are obtained via a particular type of adiabatic limit, so that in
\eqref{eq:G2-decomposition} we have \( \omega = \epsilon\omega_{0} \) with
\( h = h_{\epsilon} \), \( \psi_{c} = \psi_{c,\epsilon} \), for a parameter
\( \epsilon > 0 \), with \( \lim_{\epsilon\to0}h_{\epsilon} \equiv 1 \) and
\( \lim_{\epsilon\to0}\psi_{c,\epsilon} \) the Calabi-Yau complex volume form
of~\( B \).
Arguments of Cavalleri~\cite{Cavalleri:T2-AC}, show that if the initial Calabi-Yau
data is \( T^{2} \)-invariant, then so are the deformed symplectic
\( \SU(3) \)-structures of the solutions.

\section{Graph lifting}
\label{sec:graph-lifting}

Associated to the multi-Hamiltonian geometries above there are certain graphs
determined by the points with non-trivial stabilisers.
For toric \( G_{2} \)-manifolds these were described in
\cite{Madsen-S:toric-G2,Madsen-S:multi-large}, for toric Calabi-Yau structures
they are known in the physics literature as
\( pq \)-webs~\cite{Aharony-HK:pq-5}.
However, the existence of such graphs does not need the holonomy reduction.

\begin{proposition}
  \label{prop:graph}
  Let \( M \) be a manifold with a closed \( G_{2} \)-structure and an action
  of~\( T^{3} \) that is multi-Hamiltonian.
  Write \( M^{0} \) for the points of~\( M \) with non-trivial stabiliser.
  Then \( \Gamma_{M} = M^{0}/T^{3} \subset M/T^{3} \) is an embedded trivalent graph.

  Furthermore, the multi-moment map
  \( \nu\colon M \to \Lambda^{2}\LIE(T^{3})^{*} \) maps each edge of the graph to a
  non-trivial part of a straight line with rational tangent
  in~\( \bR^{3} = \Lambda^{2}\LIE(T^{3})^{*} \).
  At the image of each vertex, these lines satisfy a zero-tension condition.
\end{proposition}

Let us explain some of the terminology from the statement.
Graphs are to be understood in the sense of
\cite[Definiton~12]{Yeats:combinatorial-QFT}, so there may be both edges that
join two vertices and unbounded edges that meet only a single vertex.
The graph is \emph{trivalent} if exactly three distinct edges meet at each
vertex.

Recall that \( T^{3} = \bR^{3}/\bZ^{3} \).
So we can identify \( \LIE(T^{3}) \) with~\( \bR^{3} \), and \( \bZ^{3} \)~is
then the integral lattice.
To be concrete, given a \( \bZ \)-basis \( \la{U}_{1},\la{U}_{2},\la{U}_{3} \)
of \( \bZ^{3} \subset \bR^{3} = \LIE(T^{3}) \), we can consider the dual basis
\( u^{1},u^{2},u^{3} \).
This generates a lattice~\( \bZ^{3} \) in~\( \LIE(T^{3})^{*} \cong \bR^{3} \).
The target space \( \Lambda^{2}\LIE(T^{3})^{*} \cong \bR^{3} \) of~\( \nu \) also has a
lattice~\( \bZ^{3} \), generated by \( u^{2} \wedge u^{3} \),
\( u^{3} \wedge u^{1} \) and \( u^{1} \wedge u^{2} \).
These identifications are canonical if we fix an orientation
of~\( \LIE(T^{3}) \), since change of \( \bZ \)-basis is given by an element
of~\( \GL(3,\bZ) \) and \( \det(\GL(3,\bZ)) = \Set{\pm1} \).

A \emph{straight line} in \( \Lambda^{2}\LIE(T^{3})^{*} \) with \emph{rational
tangent} is an affine line of the form
\begin{equation}
  \label{eq:str-line}
  r + \bR \la{U} \hook u^{123},
\end{equation}
where \( \la{U} \in \bZ^{3} \setminus \Set{0} \),
\( r \in \Lambda^{2}(\LIE(T^{3}))^{*} \) and
\( u^{123} = u^{1} \wedge u^{2} \wedge u^{3} \).
We can choose \( \la{U} \) to be \emph{primitive}, meaning that it is not a
positive multiple of another vector in~\( \bZ^{3} \).
Note that \eqref{eq:str-line} contains the half-line
\( r + \bR_{\geqslant 0}\la{U} \hook u^{123} \).
If three half-lines given by primitive vectors \( \la{U},\la{V},\la{W} \) meet
at~\( r \), then we say there is \emph{zero-tension} at the vertex~\( r \) if
\( \la{U}+\la{V}+\la{W}=0 \).
The final claim of the Proposition is that for each vertex~\( x \) of the graph,
the three half-lines of~\( \nu(M^{0}) \) meeting at~\( r = \nu(x) \) have
zero-tension at~\( r \).

\begin{proof}[Proof of Proposition~\ref{prop:graph}]
  Consider a point \( p \in M^{0} \).
  As noted in Section~\ref{sec:special-multi-toric}, the stabiliser of~\( p \)
  is a connected subtorus of~\( T^{3} \) of dimension at most~\( 2 \).
  So there is some non-zero
  \( \la{U} = \sum_{i=1}^{3} b_{i}\la{U}_{i} \in \LIE(T^{3}) \) such that the
  generated vector field~\( U \) fixes~\( p \) and has fundamental
  period~\( 1 \).
  Let \( M_{p}^{U} \subset M^{0} \) be the connected component of the fixed-point
  set~\( M^{U} = \Set{x \in M \with U_{x} = 0} \) that contains~\( p \).
  We will show that \( \nu(M_{p}^{U}) \) is a non-trivial part of a straight line
  in \( \bR^{3} = \Lambda^{2}\LIE(T^{3})^{*} \), when \( p \) is just stabilised by a
  circle.

  Given \( \la{C} \in \Lambda^{2}\LIE(T^{3}) \), we may write
  \begin{equation*}
    \la{C} = \sum_{(jk\ell) = (123)} c_{j} \la{U}_{k} \wedge \la{U}_{\ell}.
  \end{equation*}
  The definition~\eqref{eq:d-nu} of the multi-moment map~\( \nu \) says
  \begin{equation*}
    \inp{d\nu}{\la{C}} = \sum_{(jk\ell) = (123)} c_{j} \varphi(U_{k},U_{\ell},\any)
    = \sum_{(jk\ell) = (123)} c_{j} U_{k} \wedge U_{\ell} \hook \varphi.
  \end{equation*}
  Since, for \( q \in M^{U}_{p} \) we have \( U_{q} = 0 \), we see that
  \( \inp{d\nu_{q}}{\la{C}} = 0 \) whenever there is a \( \la{V} \) such that
  \( \la{C} = \la{U} \wedge \la{V} \).
  This is the same as \( \la{C} \wedge \la{U} = 0 \).  But
  \begin{equation*}
    \la{C} \wedge \la{U} = \sum_{(ijk)=(123)} c_{i}b_{i}\, \la{U}_{j} \wedge \la{U}_{k} \wedge
    \la{U}_{i} = c^{T}b \, \la{U}_{1} \wedge \la{U}_{2} \wedge \la{U}_{3}.
  \end{equation*}
  So \( \inp{d\nu_{q}}{\la{C}} = 0 \) whenever \( c^{T}b = 0 \).
  On the other hand,
  \begin{equation*}
    (\la{U}\hook u^{123})(\la{C}) = \sum_{(ijk)=(123)} b_{i}u^{jk}(\la{C}) =
    c^{T}b,
  \end{equation*}
  so \( \ker (\la{U} \hook u^{123}) \leqslant \ker(d\nu(T_{q}M)) \) and is two-dimensional.
  Thus \( \nu(M_{p}^{U}) \) is contained in the line
  \( \nu(p) + \bR \la{U}\hook u^{123} \).

  Suppose the stabilisers of~\( p \) and \( q \) are exactly one-dimensional.
  We may choose \( \la{V},\la{W} \in \LIE(T^{3}) \) so that
  \( \la{U} \wedge \la{V} \wedge \la{W} = \la{U}_{1} \wedge \la{U}_{2} \wedge \la{U}_{3} \).
  Then \( \inp{d\nu_{q}}{\la{V} \wedge \la{W}} = \varphi_{q}(V,W,\any) \) is non-zero and
  \( \la{U} \)-invariant.
  We have \( \varphi(V,W,\any) = g(V \times W,\any) \), by the defining equation for the
  \( G_{2} \) cross-product~\( \times \).
  The vectors \( V \), \( W \) and \( V \times W \) are \( U \)-invariant and
  linearly independent at~\( q \).
  But as the \( T^{3} \)-action is effective, \( U \)~must act non-trivially on
  their four-dimensional orthogonal complement~\( Z_{q} \subset T_{q}M \).
  As \( U \) preserves the \( G_{2} \)-structure and two linearly independent
  vectors, we therefore have that \( U \) acts on~\( Z_{q} \) as a non-zero
  diagonal matrix~\( \diag(it,-it) \) in~\( \su(2) \) in the standard
  representation on \( \bC^{2} = \bR^{4} \).
  We see that the tangent space to \( M^{U}_{p} \) is spanned by
  \( V, W, V \times W \) at~\( q \), cf.~\cite[Section 4.2.3]{Madsen-S:toric-G2}, so
  \( M^{U}_{p} \) is associative.
  Consequently, \( d\nu_{q} \) has rank one on this submanifold.
  It follows that \( \nu \) is injective on the submanifold~\( M^{U}_{p} \).

  At a point~\( p \) where the stabiliser is of dimension two, this must be a
  two-torus \( T^{2} \subset T^{3} = \bR^{3}/\bZ^{3} \).
  The \( T^{3} \)-action stabilises the one-dimensional orbit of
  \( S^{1} = T^{3}/T^{2} \) through~\( p \) and preserves the
  \( G_{2} \)-structure.
  It follows that \( T^{2} \) acts on \( T_{p}M \) as \( \bR \oplus \bC^{3} \), with
  the action on \( \bC^{3} \) as the maximal torus in~\( \SU(3) \).
  We get a \( \bZ \)-basis \( \la{U}_{1},\la{U}_{2},\la{U}_{3} \) of
  \( \bZ^{3} \) such that \( \la{U}_{1},\la{U}_{2} \) span \( \LIE(T^{2}) \) and
  \( \la{U}_{3} \) generates the quotient~\( S^{1} \).
  The point \( p \) has a tubular neighbourhood that is equivariantly
  diffeomorphic to~\( S^{1} \times \bC^{3} \).
  Viewing this as a circle bundle over~\( \bC^{3} \), the discussion of
  Section~\ref{sec:special-multi-toric} shows that \( \bC^{3} \) inherits a symplectic
  form.
  By the equivariant Darboux Theorem for the \( T^{2} \)-action, this is
  symplectomorphic to the standard symplectic form on~\( \bC^{3} \).
  We may choose a basis so that \( T^{2} \) acts via the diagonal matrices
  \( D_{\theta,\phi} = \diag(e^{2\pi i\theta},e^{2\pi i\phi},e^{-2\pi i(\theta+\phi)}) \)
  in~\( \SU(3) \).
  In this neighbourhood, \( S^{1} \times \Set{0} \) has stabiliser~\( T^{2} \) and
  each \( S^{1} \times L \), for \( L \) a complex coordinate axis, has
  one-dimensional stabiliser. Thus in \( M/T^{3} \), points with two-dimensional
  stabilisers give a discrete set, and at each such point three lines,
  corresponding to points with one-dimensional stabilisers, meet.
  We thus have a trivalent graph \( \Gamma_{M} = M^{0}/T^{3} \) in \( M/T^{3} \).

  To get the zero-tension condition, write \( (z_{1},z_{2},z_{3}) \) for the
  complex coordinates on the \( \bC^{3} \)-factor in the tubular neighbourhood.
  The stabiliser acts via \( D_{\theta,\phi} \), so the Kähler moment for the
  \( T^{2} \)-action is
  \( \mu = \tfrac{1}{2}(\abs{z_{3}}^{2} - \abs{z_{1}}^{2},\abs{z_{3}}^{2} -
  \abs{z_{2}}^{2}) \), see~\cite[Section 4.1.1]{Madsen-S:toric-G2}.
  Along each coordinate axis~\( L \) we have \( U_{1} \wedge U_{2} = 0 \), so
  \( d\nu_{3} = 0 \) and \( \nu_{3} \)~is constant on~\( L \).
  The derivatives of \( \mu \) are \( (-1,0) \), \( (0,-1) \) and \( (1,1) \)
  along the \( z_{i} \)-axis, \( i=1,2,3 \), respectively.
  As these are primitive vectors that sum to zero, we get the zero-tension
  condition under \( \mu \).
  Since \( (\nu_{1},\nu_{2}) = \pi^{*}(\mu_{1},\mu_{2}) \), we also get the zero-tension
  under~\( \nu \).
\end{proof}

As the above proof shows, by considering \( S^{1} \times B \), we get a corresponding
result for symplectic \( \SU(3) \)-structures with multi-Hamiltonian
\( T^{2} \)-action.
There is then a trivalent graph \( \Gamma_{B} \) and its image under \( \mu \)
satisfies the zero-tension condition for each vertex.

\begin{remark}
  Note that if we have a torsion-free \( G_{2} \)-structure, then the fourth
  multi-moment map~\( \hat{\nu} \) has zero derivative on~\( M^{0} \), since
  \( U_{1} \wedge U_{2} \wedge U_{3} = 0 \) at each \( p \in M^{0} \).
  Thus for each component~\( \gamma \) of~\( \Gamma_{M} \), its image
  \( (\nu,\hat{\nu})(\gamma) \subset \bR^{4}\) lies in an affine hyperplane given by
  \( \hat{\nu} \in \bR \) being constant.
  Similarly, for Calabi-Yau structures on~\( B \), we have
  \( X_{1} \wedge X_{2} = 0 \) on \( B^{0} \) and each component~\( \gamma \)
  of~\( \Gamma_{B} \) has image \( (\mu,\hat{\mu})(\gamma) \) in an affine plane on which
  \( \hat{\mu} \in \bR^{2} \) is constant.

  In particular, if the graphs are connected, then when determining their
  images, it is sufficient to work in the corresponding affine subspace, and we
  only need to focus on the effect of~\( \mu \) and \( \nu \).

  Taking \( B \) to be a product \( N \times (S^{1} \times \bR) \) of a hyperKähler
  four-manifold~\( N \) obtained from the Gibbons-Hawking Ansatz
  \cite{Gibbons-H:multi} and a cylinder \( S^{1} \times \bR = \bC / \bZ \), gives a
  Calabi-Yau structure with multi-Hamiltonian \( T^{2} \)-symmetry
  \cite{Madsen-S:multi-large}.
  In this case, \( \Gamma_{B} \)~embeds in~\( \bR^{4} \) as a collection of parallel
  straight lines, corresponding to the fixed points of the circle action
  on~\( N \).
  Thus \( \Gamma_{B} \) can have an arbitrary number of connected components.
\end{remark}

The problem we now wish to study is how graphs on a base~\( B \) for an
\( \SU(3) \)-structure lift to graphs for a \( G_{2} \)-structure on a principal
circle bundle~\( M \) over~\( B \).
We assume that \( B \) has a \( T^{2} \)-symmetry preserving the
\( \SU(3) \)-structure generated
by~\( \la{X}_{1}, \la{X}_{2} \in \LIE(T^{2}) \).
If the action lifts to a \( T^{3} \)-action on a circle bundle~\( M \)
over~\( B \), then the curvature two-form \( F \in \Omega^{2}(B) \) is
\( T^{2} \)-invariant and has integral periods.
Conversely, any such two-form specifies a principal circle bundle~\( M \) and by
Lashof, May and Segal \cite{Lashof-MS:equivariant}, the \( T^{2} \)-action on~\( B \) lifts to a
\( T^{2} \)-action on~\( M \) that commutes with the principal action.
Hence we get an effective action of~\( T^{3} \), generated by
\( \la{U}_{1},\la{U}_{2} \) with \( \pi_{*}(U_{1}) = X_{2} \) and
\( \pi_{*}(U_{2}) = -X_{1} \), and by the principal action of~\( \la{U}_{3} \).

As explained in \cite[Propositions~2.1 and 2.3]{Swann:twist}, a necessary and
sufficient condition for there to be a \( T^{2} \)-invariant connection
one-form~\( \theta \) on~\( M \) with curvature~\( F \) is that there are functions
\( \lambda_{i} \) such that
\begin{equation}
  \label{eq:lambda-i}
  d\lambda_{i} = -F(X_{i},\any),
\end{equation}
for \( i=1,2 \).
These are then \( T^{2} \)-invariant, and we get \( F(X_{1},X_{2}) = 0 \).
In other words, the \( T^{2} \)-action is Hamiltonian with respect to the
(possibly degenerate) two-form~\( F \).
The generators \( U_{1},U_{2} \) are then given by
\begin{equation*}
  \tilde{X}_{1} + \lambda_{1}U_{3} = -U_{2},\quad
  \tilde{X}_{2} + \lambda_{2}U_{3} = U_{1},
\end{equation*}
where \( \tilde{X} \in \ker \theta \) denotes the horizontal lift of~\( X \)
from~\( B \), and we have written \( \lambda_{i} \) for~\( \pi^{*}\lambda_{i} \).

\begin{remark}
  Note that \eqref{eq:lambda-i} only specifies \( \lambda_{i} \) up to an additive
  constant.
  By Lashof, May and Segal \cite{Lashof-MS:equivariant}, there is some choice so that the lifted
  action is effective.
  On the other hand, for any choice of~\( \lambda_{i} \), we get the same action
  of~\( T^{3} \), just with respect to a different real basis
  of~\( \LIE(T^{3}) \).
\end{remark}

Let \( \gamma\colon I = [0,1] \to M \) be any path.
Suppose \( X \) is a constant linear combination of \( X_{1} \) and \( X_{2} \),
so \( X \) lifts to \( U = \tilde{X} + \lambda U_{3} \) on \( M \) with
\( d\lambda = - X \hook F \).  Then we have
\begin{equation}
  \label{eq:lift-diff}
  \begin{split}
    \lambda(\gamma(1)) - \lambda(\gamma(0))
    &= \int_{t=0}^{1} \Fd{}{t} \lambda(\gamma(t)) \,dt = \int_{t=0}^{1} d\lambda(\dot{\gamma}(t)) \,dt \\
    &= - \int_{t=0}^{1} F(X,\dot{\gamma}(t)) \,dt = - \int_{t=0}^{1} F(X,\dot{\tau}(t)) \,dt,
  \end{split}
\end{equation}
for \( \tau = \pi \circ \gamma \), the projection of \( \gamma \) to~\( B \).

If \( \la{X} = r_{1}\la{X}_{1} + r_{2}\la{X}_{2} \) is an integral linear
combination of \( \la{X}_{1} \) and \( \la{X}_{2} \), then \( X \) generates a
circle action on~\( B \), via a one-parameter group~\( \varphi_{s} \), with
\( \varphi_{1} = \varphi_{0} = \Id_{B} \).
The combination of \( \varphi_{s} \) and \( \tau \) defines a, perhaps singular,
parameterised surface \( S = \sigma_{X}([0,1] \times [0,1]) \subset B \) via
\( \sigma_{\la{X}}(s,t) = \varphi_{s}(\tau(t)) \).  We may rewrite \eqref{eq:lift-diff} as
\begin{equation}
  \label{eq:lift-diff-S}
  \begin{split}
    \lambda(\tau(1)) - \lambda(\tau(0))
    &= - \int_{t=0}^{1}\int_{s=0}^{1} F(X,\dot{\tau}(t)) \,ds\,dt \\
    &= - \int_{t=0}^{1}\int_{s=0}^{1} F(\partial_{s},\partial_{t}) \,ds\,dt
    = - \int_{S} F.
  \end{split}
\end{equation}

If the endpoints \( \tau(0) \) and \( \tau(1) \) are fixed points for~\( X \), then, as
a singular \( 2 \)-chain, \( \sigma_{\la{X}} \)~is a cycle.
Equation~\eqref{eq:lift-diff-S} now reads
\begin{equation}
  \label{eq:int-diff}
  \lambda(\tau(1)) - \lambda(\tau(0)) = - [F] \cap [\sigma_{\la{X}}],
\end{equation}
where \( [F] \in H^{2}(B) \) and \( [\sigma_{\la{X}}] \in H_{2}(B,\bZ) \).
As \( F \) has integral periods, \eqref{eq:int-diff} is an integer.

Similar considerations for the Kähler moment map give
\begin{equation}
  \label{eq:mu-diff}
  \inp[\big]{\mu(\tau(1))-\mu(\tau(0))}{\la{X}} = [\omega] \cap [\sigma_{\la{X}}],
\end{equation}
when \( X \) fixes \( \tau(0) \) and~\( \tau(1) \).

Suppose now the \( \SU(3) \)-structure on~\( B \) is symplectic.
Let \( e = [\mu(b),\mu(c)] \) in \( \mu(\Gamma_{B}) \) be the image of a compact edge
in~\( \Gamma_{B} \).
The corresponding points in~\( B \) have a one-dimensional stabiliser generated
by a primitive vector \( \la{X} = r_{1}\la{X}_{1} + r_{2}\la{X}_{2} \).

Let \( \la{Z} = z_{1}\la{X}_{1} + z_{2}\la{X}_{2} \) be a choice of primitive
vector so that \( \la{X}, \la{Z} \) is a \( \bZ \)-basis for \( \LIE(T^{2}) \)
with the same orientation as \( \la{X}_{1},\la{X}_{2} \), so
\begin{equation*}
  r_{1}z_{2} - r_{2}z_{1} = 1.
\end{equation*}
Then \( Z \) is a primitive generator for the action of
\( T^{2}/\Span{\exp\la{X}} \).

Let \( \tau\colon [0,1] \to B \) be a path such that \( \mu \circ \tau \) parameterises the
line segment~\( e \) affinely: \( (\mu\circ\tau)(a) = (1-a)\mu(b) + a\mu(c) \).
Then the surface \( \cE \subset B \) corresponding to the edge~\( e \) is given by
\( T^{2}\tau \) and is generated effectively by \( Z \) and~\( \tau \).
As \( X = 0 \) on~\( \tau \), we have \( [\sigma_{\la{X}}] = 0 \) in
\( H_{2}(B,\bZ) \), so
\begin{equation*}
  r_{1}[\sigma_{\la{X}_{1}}] + r_{2}[\sigma_{\la{X}_{2}}] = 0.
\end{equation*}
We find now that
\begin{equation*}
  \begin{split}
    r_{1}[\cE]
    &= r_{1}z_{1}[\sigma_{\la{X}_{1}}] + r_{1}z_{2}[\sigma_{\la{X}_{2}}]
    = r_{1}z_{1}[\sigma_{\la{X}_{1}}] + (1+r_{2}z_{1})[\sigma_{\la{X}_{2}}] \\
    &= [\sigma_{\la{X}_{2}}] + z_{1}(r_{1}[\sigma_{\la{X}_{1}}] + r_{2}[\sigma_{\la{X}_{2}}])
    = [\sigma_{\la{X}_{2}}]
  \end{split}
\end{equation*}
and similarly
\begin{equation*}
  \begin{split}
    r_{2}[\cE]
    &= r_{2}z_{1}[\sigma_{\la{X}_{1}}] + r_{2}z_{2}[\sigma_{\la{X}_{2}}]
    = (-1+r_{1}z_{2})[\sigma_{\la{X}_{1}}] + r_{2}z_{2})[\sigma_{\la{X}_{2}}] \\
    &= -[\sigma_{\la{X}_{1}}] + z_{2}(r_{1}[\sigma_{\la{X}_{1}}] + r_{2}[\sigma_{\la{X}_{2}}])
    = -[\sigma_{\la{X}_{1}}].
  \end{split}
\end{equation*}
Using \eqref{eq:mu-diff}, we get in coordinates that
\begin{equation}
  \label{eq:mu-diff-2}
  \mu(c) - \mu(b)
  = [\omega] \cap ([\sigma_{\la{X}_{1}}],[\sigma_{\la{X}_{2}}])
  = [\omega] \cap [\cE] (-r_{2},r_{1})
\end{equation}
and similarly, for \( \lambda = (\lambda_{1},\lambda_{2}) \),
\begin{equation}
  \label{eq:lambda-diff}
  \lambda(c) - \lambda(b) = - [F] \cap ([\sigma_{\la{X}_{1}}],[\sigma_{\la{X}_{2}}])
  = - [F] \cap [\cE] (-r_{2},r_{1}).
\end{equation}

Given a smooth positive function~\( h \) on~\( B \), we define a
\( T^{3} \)-invariant \( G_{2} \)-structure on the circle bundle~\( M \)
via~\eqref{eq:G2-decomposition}.
Assume that this \( G_{2} \)-geometry has \( \varphi \)~closed.
Then the multi-moment map~\( \nu \) for~\( M \) has
\( (\nu_{1},\nu_{2}) = \pi^{*}(\mu_{1},\mu_{2}) \), up to an additive constant.
We also have
\begin{equation*}
  \begin{split}
    d\nu_{3}
    &= U_{1} \wedge U_{2} \hook \varphi \\
    &= \theta(U_{1})(\pi^{*}\omega)(U_{2},\any) - \theta(U_{2})(\pi^{*}\omega)(U_{1},\any) \eqbreak
      + (\pi^{*}\omega)(U_{1},U_{2})\theta + \pi^{*}(h^{3/4}\psi_{+})(U_{1},U_{2},\any) \\
    &= -\lambda_{2}d\mu_{1} + \lambda_{1}d\mu_{2} + \pi^{*}(\omega(X_{1},X_{2}))\theta +
      \pi^{*}(h^{3/4}\psi_{+}(X_{1},X_{2},\any)).
  \end{split}
\end{equation*}
Above \( e \) we have \( X_{1} \wedge X_{2} = 0 \), and using \eqref{eq:mu-diff-2}
gives \( d\nu = (d\nu_{1},d\nu_{2},d\nu_{3}) \) is proportional to
\begin{equation*}
  (-r_{2}, r_{1}, r_{1}\lambda_{1} + r_{2}\lambda_{2}).
\end{equation*}
Note that
\( d(r_{1}\lambda_{1}+r_{2}\lambda_{2}) = -(r_{1}X_{1}+r_{2}X_{2}) \hook F = 0 \)
along~\( e \), so we get at straight line in~\( \Lambda^{2}\LIE(T^{3})^{*} \), as
expected, with tangent vector
\( (-r_{2},r_{1},r_{1}\lambda_{1}+r_{2}\lambda_{2}) \).
Thus, for some \( t \in \bR \), we have
\begin{equation*}
  \begin{split}
    \nu(c) - \nu(b)
    &= t(-r_{2}, r_{1}, r_{1}\lambda_{1}(b) + r_{2}\lambda_{2}(b)) \\
    &= (\mu_{1}(c) - \mu_{1}(b), \mu_{2}(c) - \mu_{2}(b), \nu_{3}(c) - \nu_{3}(b)),
  \end{split}
\end{equation*}
for any \( b,c \in M \) projecting to \( b,c \in B \).
Using \eqref{eq:mu-diff-2} and the first two components, we get
\begin{equation}
  \label{eq:omega-curve}
  t = [\omega] \cap [\cE],
\end{equation}
and hence
\begin{equation}
  \label{eq:nu-diff}
  \begin{split}
    \nu_{3}(c)
    &= \nu_{3}(b) + t(r_{1}\lambda_{1}(b) + r_{2}\lambda_{2}(b)) \\
    &= \nu_{3}(b) + t(r_{1}\lambda_{1}(c) + r_{2}\lambda_{2}(c)).
  \end{split}
\end{equation}

\begin{remark}
  Note that
  \begin{equation*}
    \nu(c) - \nu(b) = (x,y,\lambda_{1}(b)y - \lambda_{2}(b)x),
  \end{equation*}
  so if \( \lambda \) is replaced by \( \lambda + \lambda^{0} \), then \( \nu \) is changed by the
  linear transformation
  \begin{equation*}
    \nu \mapsto
    \begin{pmatrix}
      1&0&0 \\
      0&1&0 \\
      -\lambda^{0}_{2}&\lambda^{0}_{1}&1
    \end{pmatrix}
    \nu.
  \end{equation*}
  When determining the lifted graphs, we may therefore choose to have
  \( \lambda = 0 \) at some basepoint.
\end{remark}

\section{Toric Calabi-Yau manifolds}
\label{sec:toric-CY}

Let us now consider a situation where we can be more concrete.
We take \( B \) to be a toric Calabi-Yau manifold, so there is an action of
\( T^{3} \) preserving the complex and symplectic structures, but not the
complex volume form.
Singular examples may be obtained by considering cones
\( C(S) = \bR_{\geqslant0} \times S \) on toric Sasaki-Einstein
manifolds~\( S \) of dimension~\( 5 \); these give smooth examples, asymptotic
to the cone, by taking a crepant resolution of the singularity
\cite{Martelli-S:toric,Futaki-OW:transverse,Cho-FO:toric-SE,%
vanCoevering:examples,vanCoevering:crepant}.
Note that there is one such Calabi-Yau solution for each Kähler class on the
resolution.

For \( C(S) \) the image of the symplectic moment map is a cone on a plane
convex polygon.
Dually this is described by the fan, which is the convex set generated by rays
\( \bR_{\geqslant0}u \) of normals~\( u \) to each face of the moment cone.
Bases may be chosen so the that each \( u \) is integral with last
coordinate~\( 1 \).
The fan is then a cone on its cross-section~\( \Delta \) at height~\( 1 \), and this
cross-section is a convex polygon that is good, meaning that no edge
of~\( \Delta \) contains primitive vectors that are not endpoints of the edge.

Smooth toric resolutions~\( B \) are obtained by triangulating~\( \Delta \), so that
each vertex~\( u_{i} \) is integral, so lies in~\( \bZ^{3} \) with last
coordinate~\( 1 \), and the vertices of each triangle form an integral basis
for~\( \bZ^{3} \).

Each ray corresponds to a divisor~\( D_{i} \) in~\( B \).
These satisfy the relation
\begin{equation}
  \label{eq:D-rel}
  \sum_{i=1}^{r} D_{i}u_{i} = 0.
\end{equation}
Indeed, the Picard group of~\( B \) is generated by the \( D_{i} \) subject to
the relations~\eqref{eq:D-rel}, see Cox, Little and Schenck \cite[Theorems 4.1.3 and 4.2.1]{Cox-LS:toric}.
By \cite[Theorem 12.3.2]{Cox-LS:toric}, this Picard group is isomorphic to the
second cohomology group~\( H^{2}(B,\bZ) \).

\subsection{Quadrilateral relations}
\label{sec:quadr-relat}

Suppose we have a quadrilateral in the fan picture with vertices
\( u_{1},u_{2},v_{1},v_{2} \) such that \( v_{1}v_{2} \) is an internal edge.
Denote the corresponding divisors by \( D_{1},D_{2},A_{1},A_{2} \).
Then \( \cC = A_{1} \cap A_{2} \) is a compact curve.
The quadrilateral is a union of two triangles \( u_{i},v_{1},v_{2} \), for
\( i=1,2 \).
From smoothness, the vectors \( u_{i},v_{1},v_{2} \) form a \( \bZ \)-basis
for~\( \bZ^{3} \), so the determinants satisfy
\( \det(u_{i}\; v_{1}\; v_{2}) = -\varepsilon_{i} \)
with~\( \varepsilon_{i} \in \Set{\pm1} \).
We may change the basis of~\( \bZ^{2} \leqslant \bZ^{3} \), so that
\( v_{1} = (0,0,1) \) and \( v_{2} = (0,1,1) \).
Then the values of the determinants give
\( u_{i} = (\varepsilon_{i},a_{i},1) \), for some \( a_{i} \in \bZ \).
As \( u_{1} \) and \( u_{2} \) lie on different sides of \( v_{1}v_{2} \), we
have \( \varepsilon_{1} = -\varepsilon_{2} \).
Swapping \( u_{1} \) and \( u_{2} \), we may assume \( \varepsilon_{1} = 1 \).
The four integral vectors \( u_{i},v_{j} \) are linearly dependent, so there
exist integers \( x_{i},y_{j} \in \bZ \) such that
\begin{equation*}
    0
    = x_{1}u_{1} + x_{2}u_{2} + y_{1}v_{1} + y_{2}v_{2}
    =
      \begin{pmatrix}
        x_{1} - x_{2} \\
        x_{1}a_{1} + x_{2}a_{2} + y_{2} \\
        x_{1} + x_{2} + y_{1} + y_{2}
      \end{pmatrix}
      .
\end{equation*}
We get \( x_{2} = x_{1} \) and this common value divides \( y_{2} \), by the
second equation, and the hence also \( y_{1} \), by the final component.
We may thus divide through by \( x_{1} \) and get the quadrilateral relations
\begin{equation}
  \label{eq:quadrilateral}
  u_{1} + u_{2} + y_{1}v_{1} + y_{2}v_{2} = 0,\qquad y_{1} + y_{2} = -2.
\end{equation}
See Cox, Little and Schenck \cite[Equation 6.4.4]{Cox-LS:toric}.

Divisors \( D \) outside of the quadrilateral do not meet~\( \cC \), so the
corresponding triple intersections \( DA_{1}A_{2} \) are zero.
By \cite[Corollary 6.4.3 and Proposition 6.4.4]{Cox-LS:toric},
the relations \eqref{eq:quadrilateral} give
\begin{equation}
  \label{eq:quad-triple}
  D_{1}A_{1}A_{2} = 1 = D_{2}A_{1}A_{2},\quad A_{1}^{2}A_{2} = y_{1},\quad
  A_{1}A_{2}^{2} = y_{2}.
\end{equation}

\subsection{Closed polygons}
\label{sec:closed-polygons}

Suppose \( R_{E} = \bR_{\geqslant 0} u_{E} \) is an interior ray, with corresponding
(compact) divisor~\( E \).
The cones containing \( R_{E} \) are spanned by \( u_{E} \) together with
vectors \( u_{1},\dots,u_{m} \) that we may order anti-clockwise as in
Figure~\ref{fig:polygon}.
Write \( A_{j} \) for the divisor corresponding to~\( u_{j} \), for
\( j=1,\dots,m \).  See Example~\ref{ex:hex}, below, for an explicit example.

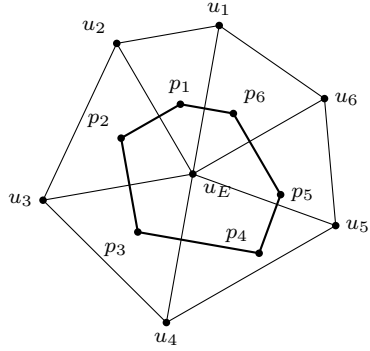
\begin{figure}
  \centering
  \begin{tikzpicture}[vertex/.style={circle,inner sep=1pt,fill},font={\scriptsize}]
    \coordinate (E) at (0,0);
    \coordinate (A1) at (80:2cm);
    \coordinate (A2) at (120:2cm);
    \coordinate (A3) at (190:2cm);
    \coordinate (A4) at (260:2cm);
    \coordinate (A5) at (-20:2cm);
    \coordinate (A6) at (30:2cm);
    \draw (E) node[vertex] {} node[below right] {\( u_{E} \)};
    \draw (A1) node[vertex] {} node[above] {\( u_{1} \)}
    -- (A2) node[vertex] {} node[above left] {\( u_{2} \)}
    -- (A3) node[vertex] {} node[left] {\( u_{3} \)}
    -- (A4) node[vertex] {} node[below] {\( u_{4} \)}
    -- (A5) node[vertex] {} node[right] {\( u_{5} \)}
    -- (A6) node[vertex] {} node[right] {\( u_{6} \)}
    -- cycle;
    \draw (E) edge (A1)
    edge (A2)
    edge (A3)
    edge (A4)
    edge (A5)
    edge (A6);

    \coordinate (P1) at ($(A1)!.5!(A2)!.5!(E)$);
    \coordinate (M2) at ($(A2)!(P1)!(E)$);
    \coordinate (P2) at ($(P1)!2.8!(M2)$);
    \coordinate (M3) at ($(A3)!(P2)!(E)$);
    \coordinate (P3) at ($(P2)!2!(M3)$);
    \coordinate (M4) at ($(A4)!(P3)!(E)$);
    \coordinate (P4) at ($(P3)!2.8!(M4)$);
    \coordinate (M5) at ($(A5)!(P4)!(E)$);
    \coordinate (P5) at ($(P4)!1.2!(M5)$);
    \coordinate (M6) at ($(A6)!(P5)!(E)$);
    \coordinate (M1) at ($(A1)!(P1)!(E)$);
    \path [name path=up] (P5) -- ($(P5)!5!(M6)$);
    \path [name path=right] (P1) -- ($(P1)!5!(M1)$);
    \draw[name intersections={of=up and right},thick]
    (P1) node[vertex] {} node[above] {\( p_{1} \)}
    -- (P2) node[vertex] {} node[above left] {\( p_{2} \)}
    -- (P3) node[vertex] {} node[below left] {\( p_{3} \)}
    -- (P4) node[vertex] {} node[above left] {\( p_{4} \)}
    -- (P5) node[vertex] {} node[right] {\( p_{5} \)}
    -- (intersection-1) node[vertex] {} node[above right] {\( p_{6} \)} -- cycle;
  \end{tikzpicture}
  \caption{A polygon \( p_{1}\dots p_{6} \) in toric picture from a fan
  \( u_{1}\dots u_{6}u_{E} \) at level~\( 1 \).}
  \label{fig:polygon}
\end{figure}

The edges in the corresponding image of the moment map are orthogonal to
differences \( u_{j} - u_{E} \), \( u_{j} - u_{k} \).
As all \( u_{i} \) have third coordinate~\( 1 \), we may regard these
differences as vectors in~\( \bZ^{2} \subset \bR^{2} \).
We may thus write the directions of the first set of edges as
\( J_{2}(u_{i}-u_{E}) \), where \( J_{2} =
\begin{psmallmatrix}
  0&-1\\
  1&0
\end{psmallmatrix}
\).
Let \( p_{j} = \mu(q_{j}) \) be the point in the image of the moment map
corresponding to the cone generated by \( u_{E}, u_{j}, u_{j+1} \), reading
indices modulo \( m \) where necessary.  By~\eqref{eq:mu-diff-2}, we have
\begin{equation*}
  p_{j} - p_{j-1} = t_{j} J_{2}(u_{j}-u_{E}),\qquad t_{j} = [\omega] \cap [\cE_{j}],
\end{equation*}
where \( \cE_{j} = A_{j} \cap E \).

\begin{proposition}
  \label{prop:closed-moment-polygon}
  \( p_{1},\dots,p_{m} \) are the vertices of a closed polygon with
  \( p_{m+1} = p_{1} \).
\end{proposition}

For the quadrilateral \( u_{j-1},u_{j+1},u_{j},u_{E} \),
equation~\eqref{eq:quadrilateral} gives
\begin{equation}
  \label{eq:quad-uE}
  u_{j-1} + u_{j+1} + a_{j}u_{j} + b_{j}u_{E} = 0,\qquad a_{j} + b_{j} = -2.
\end{equation}
Writing \( [\omega] = \sum_{i=1}^{r} w_{D_{i}}D_{i} \) and \( w_{j} = w_{A_{j}} \),
we have
\begin{equation*}
  t_{j} = \sum_{i=1}^{r} w_{D_{i}} D_{i}A_{j}E
  = w_{j-1} + w_{j+1} + a_{j}w_{j} + b_{j}w_{E},
\end{equation*}
by \eqref{eq:quad-triple}.
Now Proposition~\ref{prop:closed-moment-polygon} will follow directly from

\begin{lemma}
  \label{lem:t-u}
  In the above notation, we have
  \begin{equation*}
    \sum_{j=1}^{m} t_{j}(u_{j} - u_{E}) = 0.
  \end{equation*}
\end{lemma}

\begin{proof}
  We first compute
  \begin{equation*}
    \begin{split}
      \sum_{j=1}^{m} t_{j}(u_{j} - u_{E})
      &= \sum_{j=1}^{m} (w_{j-1} + w_{j+1} + a_{j}w_{j} + b_{j}w_{E})u_{j} \eqbreak
        - \sum_{j=1}^{m} ((2+a_{j})w_{j} + b_{j}w_{E})u_{E}.
    \end{split}
  \end{equation*}
  The coefficient of \( w_{j} \) is
  \( u_{j+1} + u_{j-1} + a_{j}u_{j} - (2+a_{j})u_{E} \), which is zero, since
  \( b_{j} = -(2+a_{j}) \).  The coefficient of \( w_{E} \) is
\begin{equation*}
  \begin{split}
        \sum_{j=1}^{m} b_{j}u_{j} - b_{j}u_{E}
      &= \sum_{j=1}^{m} -(2+a_{j})u_{j} - b_{j}u_{E} \\
        &= - \sum_{j=1}^{m} u_{j-1} + u_{j+1} + a_{j}u_{j} + b_{j}u_{E}
      = 0,
    \end{split}
  \end{equation*}
  and the assertion follows.
\end{proof}

In this context, we now have the following interpretation of the condition
\( [\omega] \cup [F] = 0 \) of Foscolo, Haskins and Nordström \cite{Foscolo-HN:from-CY}.

\begin{theorem}
  \label{thm:lift-closes}
  The lifted polygon closes, so \( \nu_{3}(q_{m+1}) = \nu_{3}(q_{1}) \), if and only
  \( [\omega] \cup [F] = 0 \).
\end{theorem}

The change in \( \nu_{3} \) is given by \eqref{eq:nu-diff}, so
\begin{equation}
  \label{eq:nu-q-diff}
  \nu_{3}(q_{j}) - \nu_{3}(q_{j-1}) = t_{j} \inp{u_{j}-u_{E}}{\lambda(q_{j})}.
\end{equation}
So we need to compute \( \lambda(q_{j}) \).
As for \( [\omega] \), we write \( [F] = \sum_{i=1}^{r}f_{D_{i}}D_{i} \) and
\( f_{j} = f_{A_{j}} \).  From \eqref{eq:lambda-diff}, we have
\begin{equation}
  \label{eq:lambda-J}
  \lambda(q_{j}) - \lambda(q_{j-1}) = - s_{j} J_{2}(u_{j} - u_{E}),
\end{equation}
where \( s_{j} = f_{j-1} + f_{j+1} + a_{j}f_{j} + b_{j}f_{E} \).

\begin{lemma}
  Putting
  \( d_{j} = f_{E}(u_{j} - u_{j+1}) + f_{j}(u_{j+1} - u_{E}) - f_{j+1}(u_{j} -
  u_{E}) \), we have
  \begin{equation*}
    \lambda(q_{j}) - \lambda(q_{1}) = J_{2} (d_{j} - d_{1}).
  \end{equation*}
\end{lemma}

\begin{proof}
  It is enough to show \( \lambda(q_{j}) - \lambda(q_{j-1}) = J_{2}(d_{j} - d_{j-1}) \).  We
  have
  \begin{equation*}
    J_{2}(\lambda(q_{j}) - \lambda(q_{j-1}))
    = s_{j}(u_{j} - u_{E}) = (f_{j-1} + f_{j+1} + a_{j}f_{j} +
    b_{j}f_{E})(u_{j}-u_{E}).
  \end{equation*}
  On the other hand
  \begin{equation*}
    \begin{split}
    d_{j} - d_{j-1}
      &= f_{E}(-u_{j-1} + 2u_{j} - u_{j+1})
        + f_{j}(u_{j+1} + u_{j-1} - 2u_{E}) \eqbreak
        - (f_{j+1} + f_{j-1})(u_{j} - u_{E})\\
        &= f_{E}((2+a_{j})u_{j} + b_{j}u_{E})
        + f_{j}(-a_{j}u_{j} - (2+b_{j})u_{E}) \eqbreak
        - (f_{j+1} + f_{j-1})(u_{j} - u_{E})\\
      &= (b_{j}f_{E} + f_{j-1}+f_{j+1} + a_{j}f_{j})(u_{E} - u_{j})
      = - J_{2}(\lambda(q_{j}) - \lambda(q_{j-1})),
    \end{split}
  \end{equation*}
  and the result follows.
\end{proof}

Now \eqref{eq:nu-q-diff} gives
\begin{equation*}
  \nu_{3}(q_{m+1}) - \nu_{3}(q_{1})
  = \sum_{j=1}^{m} t_{j} \inp{u_{j} - u_{E}}{J_{2}d_{j} + J_{2}(d_{1}) + \lambda(q_{1})}.
\end{equation*}
As the last two terms are independent of~\( j \), Lemma~\ref{lem:t-u} implies that,
after summing, their contribution is zero.  So
\begin{equation}
  \label{eq:nu-three-diff-poly}
  \begin{split}
    \MoveEqLeft[0]
    \nu_{3}(q_{m+1}) - \nu_{3}(q_{1})
    = \sum_{j=1}^{m} t_{j} \inp{u_{j} - u_{E}}{J_{2}d_{j}} \\
    &= \sum_{j=1}^{m} t_{j} \inp[\big]{u_{j} - u_{E}}{J_{2}\paren[\big]{
      (f_{E}-f_{j+1})(u_{j} - u_{E})
      + (f_{j}- f_{E})(u_{j-1} - u_{E})}} \\
    &= \sum_{j=1}^{m} t_{j}(f_{j}-f_{E}) \inp{u_{j}-u_{E}}{J_{2}(u_{j-1}-u_{E})}
    = \sum_{j=1}^{m} t_{j}(f_{j} - f_{E}),
  \end{split}
\end{equation}
since \( u_{E},u_{j-1},u_{j} \) is a positively oriented \( \bZ \)-basis.

On the other hand, \( [\omega] \cup [F] = 0 \) if and only if it pairs to zero on each
compact divisor.  We compute
\begin{equation*}
  \begin{split}
    ([\omega] \cup [F]).E
    &= \paren{w_{E}E + \sum_{j=1}^{m} w_{j}A_{j}}
      \paren{f_{E}E + \sum_{j=1}^{m} f_{j}A_{j}}
      E \\
    &= w_{E}f_{E}E^{3} + \sum_{j=1}^{m} (w_{j}f_{E} + w_{E}f_{j}) A_{j}E^{2}
      + \sum_{j,k=1}^{m} w_{j}f_{k}A_{j}A_{k}E.
  \end{split}
\end{equation*}
We know the values of the triple intersections apart from~\( E^{3} \).
But this may be computed by noting that the last component of \eqref{eq:D-rel}
gives \( E = -\sum_{D_{i} \ne E} D_{i} \), so
\( E^{3} = -\sum_{j=1}^{m} A_{j}E^{2} = - \sum_{j=1}^{m} b_{j} \).  Now
\begin{equation*}
  \begin{split}
    \MoveEqLeft ([\omega] \cup [F]).E \\
    &= -w_{E}f_{E}\sum_{j=1}^{m}b_{j} + \sum_{j=1}^{m} b_{j}(w_{j}f_{E}+w_{E}f_{j})
          + \sum_{j=1}^{m} w_{j}f_{j-1} + w_{j}f_{j+1} + a_{j}w_{j}f_{j}\\
    &= f_{E} \sum_{j=1}^{m} b_{j}(w_{j}-w_{E}) + \sum_{j=1}^{m} t_{j}f_{j}
    = -f_{E} \sum_{j=1}^{m} (2+a_{j})w_{j} + b_{j}w_{E}) + \sum_{j=1}^{m}
      t_{j}f_{j}\\
    &= \sum_{j=1}^{m} t_{j}(f_{j} - f_{E}).
  \end{split}
\end{equation*}
Thus \( ([\omega] \cup [F]).E \) agrees with \eqref{eq:nu-three-diff-poly}, completing
the proof of Theorem~\ref{thm:lift-closes}.

\section{Examples}
\label{sec:examples}

We illustrate the above theory with three examples.
The first does not require the machinery of Section~\ref{sec:toric-CY} and is tackled
more directly.

\begin{example}[The spaces \( M_{m,n} \)]
  \label{ex:Mmn}
  Let \( B = K \) be the canonical bundle of \( \CP(1)_{+} \times \CP(1)_{-} \).
  Following Foscolo, Haskins and Nordström \cite{Foscolo-HN:G2TNEH}, \( M = M_{m,n} \) denotes the total
  space over~\( K \) of the circle bundle with first Chern class
  \begin{equation*} %
    [F] = m[\omega_{+}] - n[\omega_{-}],
  \end{equation*}
  where \( [\omega_{\pm}] \) is the class of the Fubini-Study metric
  on~\( \CP(1)_{\pm} \), normalised so that
  \( \int_{\CP(1)_{\pm}}\omega_{\pm} = 1 \).  The condition \( [\omega] \cup [F] = 0 \) implies
  \begin{equation*} %
    [\omega] = k(m[\omega_{+}] + n[\omega_{-}])
  \end{equation*}
  for some \( k \in \bR \) so that \( km > 0 \) and \( kn > 0 \).

  The \( T^{2} \)-action on~\( \CP(1)_{+} \times \CP(1)_{-} \) is given by circle
  actions on each factor.
  On \( \CP(1) \) we have the action
  \( [z_{0},z_{1}] \mapsto [e^{i\theta}z_{0}, z_{1}] \) with fixed points
  \( p = [1,0], q = [0,1] \in \CP(1) \).
  Let \( \la{X}_{1} \) be the generator of the action of \( \CP(1)_{+} \) and
  \( \la{X}_{2} \) that on \( \CP(1)_{-} \).
  The \( T^{2} \)-action has fixed points
  \begin{equation*}
    \pA = (p_{+},p_{-}),\quad
    \pB = (p_{+},q_{-}),\quad
    \pC = (q_{+},p_{-}),\quad
    \pD = (q_{+},p_{+}).
  \end{equation*}
  The curves
  \begin{equation*}
    \cE_{\pA\pB} = \Set{p_{+}} \times \CP(1)_{-},\quad \cE_{\pC\pD} = \Set{q_{+}} \times
    \CP(1)_{-}
  \end{equation*}
  are fixed by~\( \la{X}_{1} \), and the element \( \la{X}_{2} \) fixes the
  curves
  \begin{equation*}
    \cE_{\pA\pC} =  \CP(1)_{+} \times \Set{p_{-}},\quad \cE_{\pB\pD} = \CP(1)_{+} \times
    \Set{q_{-}}.
  \end{equation*}
  Each of these sets maps to an edge under~\( \mu \), and there are no other
  points with non-trivial stabiliser.

  As \( \mu \), \( \lambda \), \( \nu \) are only defined up to additive constants, we
  choose them so that
  \begin{equation*}
    \mu(\pA) = (0,0),\quad \lambda(\pA) = (0,0),\quad \nu(\pA) = (0,0,0).
  \end{equation*}

  Let us now compute \( \mu(\pB) \) and \( \lambda(\pB) \).
  The points \( \pA \) and \( \pB \) are joined by \( \cE_{\pA\pB} \) with
  stabiliser \( \la{X}_{1} = 1\la{X}_{1} + 0\la{X}_{2} \), giving
  \( (r_{1},r_{2}) = (1,0) \), up to sign.  Now \( t_{\pA\pB} = [\omega] \cap
  [\cE_{\pA\pB}] = kn \) and \( s_{\pA\pB} = [F] \cap [\cE_{\pA\pB}] = n \), so
  by \eqref{eq:mu-diff-2}, \eqref{eq:lambda-diff} and \eqref{eq:nu-diff}
  \begin{equation*}
    \mu(\pB) = kn(0,1),\quad
    \lambda(\pB) = n(0,1),\quad
    \nu_{3}(\pB) = kn (1 \times 0 + 0 \times n) = 0.
  \end{equation*}

  Similarly, for \( \cE_{\pB\pD} \) the stabiliser is generated by
  \( \la{X}_{2} = 0\la{X}_{1} + 1\la{X}_{2} \) up to sign.
  We choose \( (r_{1},r_{2}) = (0,-1) \), so that \( t_{\pB\pD} = km \), \(
  s_{\pB\pD} = m \) and
  \begin{gather*}
    \mu(\pD) = \mu(\pB) + km(1,0) = (km,kn),\quad
    \lambda(\pD) = \lambda(\pB) - (m,0) = (-m,n),\\
    \nu_{3}(\pD) = \nu_{3}(\pB) + km(0 \times -m + (-1) \times n) = -kmn.
  \end{gather*}
  Similarly, \( \nu(\pC) = (km,0,0) \).
  So the lifted graph is as in Figure~\ref{fig:Mmn}, cf.~\cite{Dixon:G2-ALC-talk}.

  \begin{figure}
    \begin{equation*}
      \begin{tikzpicture}[x ={(2cm,0)}, y = {(1.4cm,1.4cm)}, z = {(0,1cm)},
        font={\scriptsize}]
        \coordinate (P1) at (0,0,0);
        \coordinate (P2) at (1,0,0);
        \coordinate (P3) at (0,1,0);
        \coordinate (P4) at (1,1,-1);
        \fill[green!20, fill opacity=.3] (P2) -- (P3) --
        (P4) -- (P2);
        \fill[red!60, fill opacity=.3] (P1) -- (P3)
        -- (1,1,0) -- (P2) -- (P1);
        \begin{scope}[every node/.style = {circle, inner sep=2pt, fill}]
          \node at (P1) {};
          \node at (P2) {};
          \node at (P3) {};
          \node at (P4) {};
        \end{scope}
        \node[above left] at (P1) {\( \nu(\pA) \)};
        \node[below] at (P2) {\( \nu(\pC) \)};
        \node[above right] at (P3) {\( \nu(\pB) \)};
        \node[below] at (P4) {\( \nu(\pD) \)};
        \coordinate (Q1) at ($(P1)+(-.2,-.2,0)$);
        \coordinate (Q2) at ($(P2)+(1.3,-1.3,1.3)$);
        \coordinate (Q3) at ($(P3)+(-.4,.4,.4)$);
        \coordinate (Q4) at ($(P4)+(.2,.2,-.4)$);
        \draw[thick] (P1) edge (Q1)
        -- (P2) edge (Q2)
        -- (P4) edge (Q4)
        -- (P3) edge (Q3)
        -- (P1);
        \node[below] at ($(P1)!.5!(P2)$) {\( km \)};
        \node[above left] at ($(P1)!.5!(P3)$) {\( kn \)};
        \draw[dotted] (P2) -- (P3) -- (1,1,0) -- cycle;
        \draw[dashed] (P4) -- node[midway,right] {\( kmn \)} (1,1,0);
        \node[right] at (1,1,0) {\( \mu(\pD) \)};
      \end{tikzpicture}
    \end{equation*}
    \caption{The \( G_{2} \)-trivalent graph for \( M_{m,n} \), with the red
    region giving the compact part of underlying planar \( \SU(3) \)-graph.}
    \label{fig:Mmn}
  \end{figure}
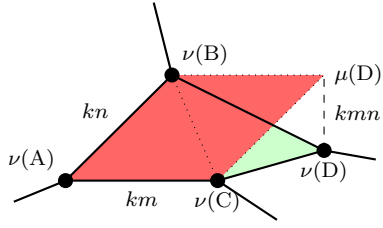
\end{example}

\begin{example}[Hexagon, star triangulation]
  \label{ex:hex}
  The fan for this example is given on the left in Figure~\ref{fig:hexagon}, with the
  vectors \( u_{1},u_{2},\ldots, u_{6},u_{E} \in \mathbb{Z}^2 \) chosen so that
  \begin{gather*}
    u_{1} - u_{E} = (-1,-1),\quad u_{2} - u_{E} = (0,-1),\quad u_{3} - u_{E} =
    (1,0),\quad \\
    u_{4} - u_{E} = (1,1),\quad u_{5} = (0,1) - u_{E},\quad u_{6} - u_{E} = (-1,0).
  \end{gather*}
  The second cohomology is generated by the divisors modulo \eqref{eq:D-rel}, so
  each class may be specified as linear combination of four divisors.
  Following the approach of Acharya et al.\ \cite{Acharya-FNS:G2-conifolds}, we choose
  \( D_{1},D_{2},D_{3},D_{4} \) as our generating set and put
  \begin{equation*} [F] = f_{1}D_1 + f_{2}D_{2} + f_{3}D_{3} + f_{4}D_{4},\quad [\omega]
    = w_{1}D_{1} + w_2D_{2} + w_{3}D_{3} + w_{4}D_{4}.
  \end{equation*}

  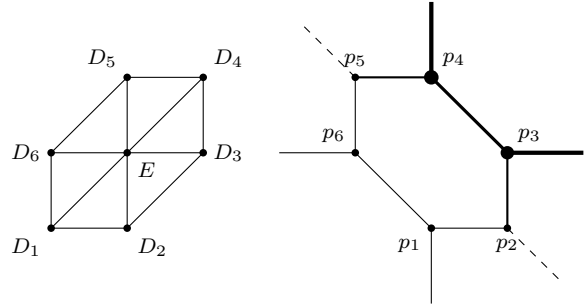
\begin{figure}
    \centering
    \begin{tikzpicture}[vertex/.style={circle,inner
      sep=1pt,fill},font={\scriptsize}]
      \begin{scope}
        \coordinate (D1) at (-1,-1);
        \coordinate (D2) at (0,-1);
        \coordinate (D3) at (1,0);
        \coordinate (D4) at (1,1);
        \coordinate (D5) at (0,1);
        \coordinate (D6) at (-1,0);
        \coordinate (E) at (0,0);
        \draw (D1) node[vertex] {} node[below left] {\( D_{1} \)}
        -- (D2) node[vertex] {} node[below right] {\( D_{2} \)}
        -- (D3) node[vertex] {} node[right] {\( D_{3} \)}
        -- (D4) node[vertex] {} node[above right] {\( D_{4} \)}
        -- (D5) node[vertex] {} node[above left] {\( D_{5} \)}
        -- (D6) node[vertex] {} node[left] {\( D_{6} \)}
        -- cycle;
        \draw (E) node[vertex] {} node[below right] {\( E \)};
        \draw (E)
        edge (D1)
        edge (D2)
        edge (D3)
        edge (D4)
        edge (D5)
        edge (D6);
      \end{scope}
      \begin{scope}[xshift=4cm,yshift=-1cm,vertex-circle/.style={circle,inner
        sep=1.5pt,draw,fill}]
        \coordinate (P1) at (0,0);
        \coordinate (P2) at (1,0);
        \coordinate (P3) at (1,1);
        \coordinate (P4) at (0,2);
        \coordinate (P5) at (-1,2);
        \coordinate (P6) at (-1,1);
        \draw (P1) edge +(0,-1);
        \draw (P2) edge[dashed] +(.7,-.7);
        \draw (P3) edge[ultra thick] +(1,0);
        \draw (P4) edge[ultra thick] +(0,1);
        \draw (P5) edge[dashed] +(-.7,.7);
        \draw (P6) edge +(-1,0);
        \draw (P5) node[vertex] {} node[above] {\( p_{5} \)}
        -- (P6) node[vertex] {} node[above left] {\( p_{6} \)}
        -- (P1) node[vertex] {} node[below left] {\( p_{1} \)}
        -- (P2) node[vertex] {} node[below] {\( p_{2} \)};
        \draw[thick] (P5) -- (P4);
        \draw[very thick] (P3)
        -- (P4) node[vertex-circle] {} node[above right] {\( p_{4} \)};
        \draw[thick] (P2)
        -- (P3) node[vertex-circle] {} node[above right] {\( p_{3} \)};
      \end{scope}
    \end{tikzpicture}
    \caption{Fan diagram for the star triangulation of the hexagon, and the
    corresponding toric diagram.}
    \label{fig:hexagon}
  \end{figure}

  Triple product computations, using the relations
  \eqref{eq:quadrilateral} and \eqref{eq:quad-triple},
  cf.\ \cite{Acharya-FNS:G2-conifolds,Xie-Y:3SCFT}, give the volumes of the six
  compact curves \( \cE_{j} = A_{j} \cap E \):
  \begin{gather*}
    t_{1} = w_{2} - w_{1},\quad
    t_{2} = w_{1} - w_{2} + w_{3},\quad
    t_{3} = w_{2} - w_{3} + w_{4},\\
    t_{4} = w_{3} - w_{4},\quad
    t_{5} = w_{4},\quad
    t_{6} = w_{1}.
  \end{gather*}
  From these we see that \( [\omega] \) is a Kähler class if
  \begin{equation*}
    w_{1} > 0,\qquad w_{4} > 0,\qquad
    w_{3} > w_{2} - w_{1} > 0,\qquad
    w_{2} > w_{3} - w_{4} > 0.
  \end{equation*}

  We choose our additive constants so that
  \begin{equation*}
    p_1 = \mu(q_1) = (0,0), \quad \lambda(q_1) = (0,0), \quad \nu(q_{1}) = (0,0,0).
  \end{equation*}
  Using the values of~\( t_{j} \), we get the remaining five vertices,
  \(p_{i} = \mu(q_{i}) \), of the planar hexagon:
  \begin{gather*}
    p_{2} = (w_{1} - w_{2} + w_{3}, 0),\quad
    p_{3} = (w_{1} - w_{2} + w_{3}, w_{2} - w_{3} + w_{4}),\\
    p_{4} = (w_{1} - w_{2} + w_{4}, w_{2}),\quad
    p_{5} = (w_{1} - w_{2}, w_{2}),\quad
    p_{6} = (w_{1} - w_{2}, w_{2} - w_{1}).
  \end{gather*}

  Next, we compute the lifts, using \eqref{eq:nu-q-diff}.
  First, we need to compute the numbers \( \lambda(q_i) \).
  In order to express these in terms of the coefficients \( f_1,\ldots,f_4 \), we use
  the formula \eqref{eq:lambda-diff} in the form~\eqref{eq:lambda-J}.
  Here \( s_{j} = [F] \cap [\cE_{j}] \) and this is computed in the same way
  as~\( t_{j} \), just with \( f_{i} \) replacing each \( w_{i} \).
  We therefore find
  \begin{gather*}
    \lambda(q_{2}) = (-f_{1} + f_{2} - f_{3}, 0),\quad
    \lambda(q_{3}) = (-f_{1} + f_{2} - f_{3}, -f_{2} + f_{3} - f_{4}),\\
    \lambda(q_{4}) = (-f_{1} + f_{2} - f_{4}, -f_{2}),\quad
    \lambda(q_{5}) = (-f_{1} + f_{2}, -f_{2}),\\
    \lambda(q_{6}) = (-f_{1} + f_{2}, -f_{1} -f_{2}).
  \end{gather*}
  Based on this, we compute the lifts \( \nu_3(q_{j}) \) using
  \eqref{eq:nu-q-diff}
  \begin{gather*}
    \nu_{3}(q_{2}) = 0,\quad
    \nu_{3}(q_{3}) = (-f_{1} + f_{2} - f_{3})(w_{2} - w_{3} + w_{4}),\\
    \nu_{3}(q_{4}) = - f_{1}w_{2} + (f_{2}-f_{3})(w_{2} - w_{3} + w_{4})  -
    f_{4}(w_{3} - w_{4}),\\
    \nu_{3}(q_{5}) = - f_{1}w_{2} + f_{2}(w_{2} - w_{3}) - f_{3}(w_{2} - w_{3} +
    w_{4}) - f_{4}(w_{3} - w_{4}),\\
    \begin{split}
      \nu_{3}(q_{6})
      &= f_{1}(w_{1} - w_{2}) - f_{2}(w_{1} - w_{2} + w_{3})
        - f_{3}(w_{2} - w_{3} + w_{4}) \eqbreak
        - f_{4}(w_{3} - w_{4}).
    \end{split}
  \end{gather*}
  In this case, the explicit expression for \( [\omega] \cup [F] = 0 \) reads
  \begin{equation*}
    \begin{split}
      f_{1}(w_{2} - w_{1}) + f_{2}(w_{1} - w_{2} + w_{3})
      + f_{3}(w_{2} - w_{3} + w_{4})
      + f_{4}(w_{3} - w_{4}) = 0.
    \end{split}
  \end{equation*}
  Note that this is a linear condition in the \( f_{i} \).
  However each \( f_{i} \) is integral, so this is also a discrete constraint on
  the~\( w_{i} \), and certain choices of \( [\omega] \) will give no non-zero
  solution.
  Also note that each coefficient of \( f_{i} \) is positive, so at least two of
  the \( f_{i} \) have opposite signs.

  With \( w_{1} = w_{4} = 1 \), \( w_{2} = w_{3} = 2 \),
  \( f_{1} = f_{2} = 1 = -f_{3} = -f_{4} \) we obtain the toric diagram on the
  right in Figure~\ref{fig:hexagon}.
  The thickness of lines and points indicate the relative values of
  \( \nu_{3} \).
  The half-lines out of \( p_{1} \) and \( p_{6} \) line in the plane of
  \( p_{5}p_{6}p_{1}p_{2} \), those out of \( p_{2} \) and \( p_{6} \) descend
  from this plane.
  The points \( p_{3},p_{4} \) are higher than this plan, and the half lines out
  of \( p_{3} \) and \( p_{4} \) rise out of a parallel plane through
  \( p_{3}p_{4} \).
  Other values of \( w_{i} \) give more irregular figures, with different
  lengths for \( p_{1}p_{2} \), \( p_{2}p_{3} \) and \( p_{3}p_{4} \).
  Varying the \( f_{j} \) also enables moving \( p_{5} \) out of the
  \( p_{6}p_{1}p_{2} \)-plane, making \( p_{3}p_{4} \) skew relative to
  \( p_{1}p_{6} \).
\end{example}

\begin{remark}
  According to Altmann~\cite{Altmann:Gorenstein}, the hexagon singularity admits
  two types of versal deformation: one into three skew lines and one into two
  triangles.
  We expect that in such cases the graph \( \Gamma_{B} \) from such a deformation has
  two separate components lying in distinct parallel planes.
  A further example in \cite{Altmann:Gorenstein}, based on an octahedron, has
  versal deformations that admit non-trivial circle bundles, and again we expect
  the \( G_{2} \)-graphs~\( \Gamma_{M} \) to have components lying in distinct
  parallel affine subspaces.
\end{remark}

\begin{example}[Second resolution of the three-fold quotient of the conifold]
  As in Acharya et al.\ \cite{Acharya-FNS:G2-conifolds}, let us consider the
  \( \bZ/3\bZ \)-quotient of the conifold.  The rays of the fan are given by
  \begin{gather*}
    u_{2} - u_{1} = (1,0),\quad u_{3} - u_{1} = (0,3),\quad u_{4} - u_{1} = (-1,3),\\
    v_{1} - u_{1} = (0,1),\quad v_{2} - u_{1} = (0,2),
  \end{gather*}
  with \( u_{i} \) corresponding to divisors \( D_{i} \), and \( v_{j} \) to
  compact divisors \( E_{j} \).
  There are two resolutions as shown in Figure~\ref{fig:Z3-conifold}.
  We will work with the second resolution, whose corresponding toric diagram is
  given in Figure~\ref{fig:mmZ3}.
  \begin{figure}
    \centering
    \begin{tikzpicture}[vertex/.style={circle,inner
      sep=1pt,fill},font={\scriptsize}]
      \begin{scope}
        \coordinate (D1) at (0,0);
        \coordinate (D2) at (1,0);
        \coordinate (D3) at (0,3);
        \coordinate (D4) at (-1,3);
        \coordinate (E1) at (0,1);
        \coordinate (E2) at (0,2);
        \draw (D1) node[vertex] {} node[below] {\( D_{1} \)}
        -- (D2) node[vertex] {} node[below right] {\( D_{2} \)}
        -- (D3) node[vertex] {} node[above] { \( D_{3} \) }
        -- (D4) node[vertex] {} node[above left] {\( D_{4} \)}
        -- cycle;
        \draw (E1) node[vertex] {} node[right] {\( E_{1} \)};
        \draw (E2) node[vertex] {} node[left] {\( E_{2} \)};
      \end{scope}
      \begin{scope}[xshift=3cm]
        \coordinate (D1) at (0,0);
        \coordinate (D2) at (1,0);
        \coordinate (D3) at (0,3);
        \coordinate (D4) at (-1,3);
        \coordinate (E1) at (0,1);
        \coordinate (E2) at (0,2);
        \draw (D1) node[vertex] {} node[below] {\( D_{1} \)}
        -- (D2) node[vertex] {} node[below right] {\( D_{2} \)}
        -- (D3) node[vertex] {} node[above] { \( D_{3} \) }
        -- (D4) node[vertex] {} node[above left] {\( D_{4} \)}
        -- cycle;
        \draw (D1)
        -- (E1) node[vertex] {} node[right] {\( E_{1} \)}
        -- (E2) node[vertex] {} node[left] {\( E_{2} \)}
        -- (D3);
        \draw (D2) -- (E1) -- (D4);
        \draw (D2) -- (E2) -- (D4);
      \end{scope}
      \begin{scope}[xshift=6cm]
        \coordinate (D1) at (0,0);
        \coordinate (D2) at (1,0);
        \coordinate (D3) at (0,3);
        \coordinate (D4) at (-1,3);
        \coordinate (E1) at (0,1);
        \coordinate (E2) at (0,2);
        \draw (D1) node[vertex] {} node[below] {\( D_{1} \)}
        -- (D2) node[vertex] {} node[below right] {\( D_{2} \)}
        -- (D3) node[vertex] {} node[above] { \( D_{3} \) }
        -- (D4) node[vertex] {} node[above left] {\( D_{4} \)}
        -- cycle;
        \draw (D1)
        -- (E1) node[vertex] {} node[below right=1ex and -.5ex] {\( E_{1} \)};
        \draw (E2) node[vertex] {} node[above left=1ex and -.5ex] {\( E_{2} \)}
        -- (D3);
        \draw (D2) -- (E1) -- (D4);
        \draw (D2) -- (E2) -- (D4);
        \draw (D2) -- (D4);
      \end{scope}
    \end{tikzpicture}
    \caption{The fan diagrams of \( \bZ/3\bZ \)-quotient of the conifold (left)
    and its two resolutions.}
    \label{fig:Z3-conifold}
  \end{figure}
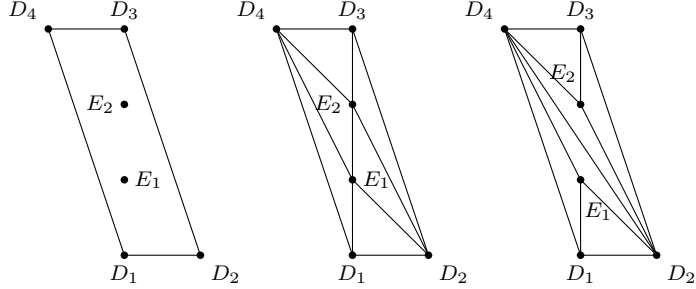

  \begin{figure}
    \centering
    \begin{tikzpicture}[vertex/.style={circle,inner
      sep=1pt,fill},font={\scriptsize}]
      \coordinate (P1) at (-.5,1.5);
      \coordinate (P2) at ($(P1)+.5*(3,2)$);
      \coordinate (P3) at ($(P1)+.5*(-1,-1)$);
      \coordinate (P4) at ($(P3)+.5*(-1,0)$);
      \coordinate (P5) at ($(P2)+.5*(1,1)$);
      \coordinate (P6) at ($(P5)+.5*(1,0)$);
      \draw (P1) node[vertex] {} node[above] {\( p_{1} \)}
      -- (P2) node[vertex] {} node[below] {\( p_{2} \)};
      \draw (P1)
      -- (P3) node[vertex] {} node[below] {\( p_{3} \)}
      -- (P4) node[vertex] {} node[left] {\( p_{4} \)}
      -- cycle
      ;
      \draw (P2)
      -- (P5) node[vertex] {} node[above] {\( p_{5} \)}
      -- (P6) node[vertex] {} node[right] {\( p_{6} \)}
      -- cycle;
    \end{tikzpicture}
    \caption{Compact part of the moment graph, before lifting, corresponding to
    the second resolution of the \( \bZ/\bZ3 \)-quotient of the conifold.}
    \label{fig:mmZ3}
  \end{figure}
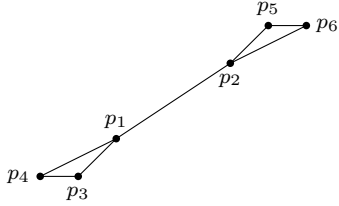
  The space of divisors is spanned by \( D_{1},D_{2},D_{3} \), so we write
  \( [\omega] = w_{1}D_{1} + w_{2}D_{2} + w_{3}D_{3} \).
  This defines a Kähler structure if \( w_{1} > - w_{2} > 0 \) and
  \( w_{3} > -w_{2} > 0 \).
  Writing \( [F] = f_{1}D_{1} + f_{2}D_{2} + f_{3}D_{3} \), the condition
  \( [\omega] \cup [F] = 0 \) is
  \begin{equation*}
    (f_{1} + f_{2})(w_{1} + w_{2}) = 0 = (f_{2} + f_{3})(w_{2} + w_{3}),
  \end{equation*}
  so \( f_{1} = -f_{2} = f_{3} \).

  Putting \( p_{1} = 0 \), we have
  \begin{equation*}
    p_{2} = -w_{2}(3,2),\quad
    p_{3} = -(w_{1}+w_{2})(1,1),\quad
    p_{6} = (2w_{3}-w_{2},w_{3}-w_{2}).
  \end{equation*}
  Similarly, we compute
  \begin{gather*}
    \lambda(p_{2}) - \lambda(p_{1}) = f_{2}(3,2),\quad
    \lambda(p_{3}) - \lambda(p_{1}) = (f_{1}+f_{2})(1,1) = (0,0),\\
    \lambda(p_{6}) - \lambda(p_{2}) = -(f_{2}+f_{3})(2,1) = (0,0),
  \end{gather*}
  and so, if \( \lambda(p_{1}) = 0 \), we get
  \begin{gather*}
    \nu_{3}(p_{2}) - \nu_{3}(p_{1}) = 0,\quad
    \nu_{3}(p_{3}) - \nu_{3}(p_{1}) = 0,\\
    \nu_{3}(p_{6}) - \nu_{3}(p_{2}) = f_{2}(w_{2}+w_{3}) \ne 0.
  \end{gather*}
  So the lifted graph is not planar, with the two triangles rotated relative to
  each other along the central axis.
\end{example}

\providecommand{\bysame}{\leavevmode\hbox to3em{\hrulefill}\thinspace}
\providecommand{\MR}{\relax\ifhmode\unskip\space\fi MR }
\providecommand{\MRhref}[2]{%
  \href{http://www.ams.org/mathscinet-getitem?mr=#1}{#2}
}
\providecommand{\href}[2]{#2}

\end{document}